\documentclass[12pt]{article}
 \usepackage{amsmath}
 \usepackage{amssymb}
\usepackage{amsfonts}
\usepackage{amsbsy}
\usepackage[small,nohug,heads=littlevee]{diagrams}
\diagramstyle[labelstyle=\scriptstyle]
\newenvironment{proof}[1][Proof]{\begin{trivlist}
\item[\hskip \labelsep {\bfseries #1}]}{\end{trivlist}}
    \newcommand{\qed}{\nobreak \ifvmode \relax \else
          \ifdim\lastskip<1.5em \hskip-\lastskip
          \hskip1.5em plus0em minus0.5em \fi \nobreak
          \vrule height0.75em width0.5em depth0.25em\fi}
\newtheorem{myth}{Theorem}[section]
\newtheorem{mylem}{Lemma}[section]
\newtheorem{mycor}{Corollary}[section]

\newtheorem{mydef}{Definition}[section]

\newtheorem{myrem}{Remark}[section]
\newtheorem{myexam}{Example}[section]
\begin{document}
\date{To Professor ART Solarin on his 60th Birthday Celebration}
\title{Holomorph of generalized Bol loops II
}
\author{T. G. Jaiy\'e\d ol\'a\thanks{All correspondence to be addressed to this author.} \\
Department of Mathematics,\\
Faculty of Science,\\
Obafemi Awolowo University,\\
Ile Ife 220005, Nigeria.\\
jaiyeolatemitope@yahoo.com\\tjayeola@oauife.edu.ng \and
B. A. Popoola\\
Department of Mathematics,\\
Federal College of Education,\\
Osiele, Abeokuta 110101, Nigeria.\\
bolajipopoola2002@yahoo.com}\maketitle
\begin{abstract}
The notion of the holomorph of a generalized Bol loop (GBL) is characterized afresh. The holomorph of a right inverse property loop (RIPL) is shown to be a GBL if and only if the loop is a GBL and some bijections of the loop are right (middle) regular. The holomorph of a RIPL is shown to be a GBL if and only if the loop is a GBL and some elements of the loop are right (middle) nuclear. Necessary and sufficient condition for the holomorph of a RIPL to be a Bol loop are deduced. Some algebraic properties and commutative diagrams are established for a RIPL whose holomorph is a GBL.
\end{abstract}
\section{Introduction}
\paragraph{}
Let $L$ be a non-empty set. Define a binary operation ($\cdot $) on
$L$ : If $x\cdot y\in L$ for all $x, y\in L$, $(L, \cdot )$ is
called a groupoid. If the equations:
\begin{displaymath}
a\cdot x=b\qquad\textrm{and}\qquad y\cdot a=b
\end{displaymath}
have unique solutions for $x$ and $y$, respectively, for each $a,b\in L$, then $(L, \cdot
)$ is called a quasigroup. For each $x\in L$, the elements $x^\rho
=xJ_\rho\in L$ and $x^\lambda =xJ_\lambda\in L$ such that
$xx^\rho=e^\rho$ and $x^\lambda x=e^\lambda$ are called the right
and left inverse elements of $x$ respectively. Here, $e^\rho\in L$
and $e^\lambda\in L$ satisfy the relations $xe^\rho =x$ and
$e^\lambda x=x$ for all $x\in L$ if they exist in a quasigroup $(L, \cdot
)$ and are respectively called the
right and left identity elements. Now, if $e^\rho=e^\lambda=e\in L$, then $e$ is called the identity element and $(L, \cdot )$ is called a loop. In case $x^\lambda =x^\rho$, then, we simply write $x^\lambda =x^\rho =x^{-1}=xJ$ and refer to $x^{-1}$ as the inverse of $x$. If $x,y,z\in L$ such that $(x\cdot yz)=(xy\cdot z)(x,y,z)$, then $(x,y,z)$ is called the associator of $x,y,z$.

Let $x$ be an arbitrarily fixed element in a loop $(G, \cdot )$. For any $y\in G$, the left and
right translation maps of $x\in G$, $L_x$ and $R_x$ are respectively
defined by
\begin{displaymath}
yL_x=x\cdot y\qquad\textrm{and}\qquad yR_x=y\cdot x
\end{displaymath}
A loop $(L,\cdot)$ is called
a (right) Bol loop if it satisfies the identity
\begin{equation}\label{eq:7}
(xy\cdot z)y=x(yz\cdot y)
\end{equation}
A loop $(L,\cdot)$ is called
a left Bol loop if it satisfies the identity
\begin{equation}\label{eq:7.1}
y(z\cdot yx)=(y\cdot zy)x
\end{equation}
A loop $(L,\cdot)$ is called
a Moufang loop if it satisfies the identity
\begin{equation}\label{eq:7.2}
(xy)\cdot (zx)=(x\cdot yz)x
\end{equation}
A loop $(L,\cdot)$ is called a right inverse property loop (RIPL) if it satisfies right inverse property (RIP)
\begin{equation}\label{eq:9}
(yx)x^{\rho}=y
\end{equation}
A loop $(L,\cdot)$ is called a left inverse property loop (LIPL) if it satisfies left inverse property (LIP)
\begin{equation}\label{eq:9.1}
x^\lambda(xy)=y
\end{equation}
A loop $(L,\cdot)$ is called an automorphic inverse property loop (AIPL) if it satisfies automorphic inverse property (AIP)
\begin{equation}\label{eq:10}
(xy)^{-1}=x^{-1}y^{-1}
\end{equation}

A loop $(L,\cdot)$ in which the mapping $x\mapsto x^2$ is a permutation, is called a Bruck loop if it is both a Bol loop and either AIPL or obeys the identity $xy^2\cdot x=(yx)^2$. (Robinson \cite{25})

Let $(L,\cdot)$ be a loop with a single valued self-map $\sigma:x\longrightarrow\sigma(x)$:

$(L,\cdot )$ is called a $\sigma$-generalized (right) Bol loop or right B-loop if it satisfies
the identity
\begin{equation}\label{eq:8}
(xy\cdot z)\sigma(y)=x(yz\cdot\sigma(y))
\end{equation}
$(L,\cdot )$ is called a $\sigma$-generalized left Bol loop or left B-loop if it satisfies
the identity
\begin{equation}\label{eq:8.1}
\sigma(y)(z\cdot yx)=(\sigma(y)\cdot zy)x
\end{equation}
$(L,\cdot )$ is called a $\sigma$-M-loop if it satisfies
the identity
\begin{equation}\label{eq:8.3}
(xy)\cdot (z\sigma(x))=(x\cdot yz)\sigma(x)
\end{equation}

Let $(G,\cdot )$ be a groupoid (quasigroup, loop) and let $A,B$ and $C$ be three bijective
mappings, that map $G$ onto $G$. The identity map on $G$ shall be denoted by $I$. The triple $\alpha =(A,B,C)$ is
called an autotopism of $(G,\cdot )$ if and only if
\begin{displaymath}
xA\cdot yB=(x\cdot y)C~\forall~x,y\in G.
\end{displaymath}
Such triples form a group
$AUT(G,\cdot )$ called the autotopism group of $(G,\cdot )$.

If $A=B=C$, then $A$ is called an automorphism of the
groupoid(quasigroup, loop) $(G,\cdot )$. Such bijections form a
group $AUM(G,\cdot )$ called the automorphism group of $(G,\cdot )$. Let $G$ and $H$ be groups such that $\varphi:G\to H$ is an isomorphism. If $\varphi (g)=h$, then this would be expressed as $g\overset{\varphi}{\approxeq}h$.

Given any two sets $X$ and $Y$. The statement '$f~:X\rightarrow Y$ is defined as $f(x)=y,~x\in X,~y\in Y$' will be at times be expressed as '$f~:X\rightarrow Y~\uparrow~f(x)=y$'.

The right nucleus of $(L,\cdot)$ is defined by $N_\rho (L,\cdot)=\{x\in L~|~zy\cdot x=z\cdot yx~\forall~y,z\in L\}$.
The middle nucleus of $(L,\cdot)$ is defined by $N_\mu (L,\cdot)=\{x\in L~|~zx\cdot y=z\cdot xy~\forall~y,z\in L\}$.

\begin{mydef}
Let $(G,\cdot )$ be a quasigroup. Then
\begin{enumerate}
\item a bijection $U$ is called autotopic if there exists $(U,V,W)\in AUT(G,\cdot )$; the set of all such mappings forms a group $\Sigma(G,\cdot )$.
\item a bijection $U$ is called $\rho$-regular if there exists $(I,U,U)\in AUT(G,\cdot )$; the set of all such mappings forms a group $\mathcal{P}(G,\cdot )$.
\item a bijection $U$ is called $\mu$-regular if there exists a bijection $U'$ such that $(U,U'^{-1},I)\in AUT(G,\cdot )$. $U'$ is called the adjoint of $U$. The set of all $\mu$-regular mappings forms a group $\Phi(G,\cdot )\le\Sigma(G,\cdot )$. The set of all adjoint mapping
    forms a group $\Psi(G,\cdot )$.
\end{enumerate}
\end{mydef}

\begin{mydef}
Let $(Q,\cdot)$ be a loop and $A(Q)\le AUM(Q,\cdot)$
be a group of automorphisms of the loop $(Q,\cdot)$. Let $H=A(Q)\times Q$. Define $\circ$ on $H$ as
\begin{displaymath}
(\alpha,x)\circ(\beta,y)=(\alpha\beta,x\beta\cdot y)~\textrm{for all}~(\alpha,x),(\beta,y)\in H.
\end{displaymath}
$(H,\circ)$ is a loop and is called the A-holomorph of $(Q,\cdot)$.
\end{mydef}
The left and right translations maps of an element $(\alpha,x)\in H$ are respectively
denoted by $\mathbb{L}_{(\alpha,x)}$ and $\mathbb{R}_{(\alpha,x)}$.

\begin{myrem}
$(H,\circ)$ has a subloop $\{I\}\times Q$ that is isomorphic to $(Q,\cdot)$. As observed in Lemma 6.1 of Robinson \cite{25}, given a loop $(Q,\cdot)$ with an A-holomorph $(H,\circ )$, $(H,\circ )$ is a Bol loop if and only if $(Q,\cdot)$ is a $\theta$-generalized Bol loop for all $\theta\in A(Q)$. Also in Theorem 6.1 of Robinson \cite{25}, it was shown that $(H,\circ )$ is a Bol loop if and only if $(Q,\cdot)$ is a Bol loop and $x^{-1}\cdot x\theta\in N_\rho(Q,\cdot)$ for all $\theta\in A(Q)$.
\end{myrem}

The birth of Bol loops can be traced back to Gerrit Bol \cite{16} in 1937 when he established the relationship between Bol loops and Moufang loops, the latter which was discovered by Ruth Moufang \cite{19}. Thereafter, a theory of Bol loops was evolved through the Ph.D. thesis of Robinson \cite{25} in 1964 where he studied the algebraic properties of Bol loops, Moufang loops and Bruck loops, isotopy of Bol loop and some other notions on Bol loops. Some later results on Bol loops and Bruck loops can be found in Bruck \cite{7}, Solarin \cite{38}, Adeniran and Akinleye \cite{1}, Bruck \cite{8}, Burn \cite{bur}, Gerrit Bol \cite{16}, Blaschke and Bol \cite{6}, Sharma \cite{30,31}, Adeniran and Solarin \cite{2}. In the 1980s, the study and construction of finite Bol loops caught the attention of many researchers among whom are Burn \cite{bur,bur1,bur2}, Solarin and Sharma \cite{32.1,34,35} and others like Chein and Goodaire \cite{gg22,chein1,chein2}, Foguel at. al. \cite{phd127}, Kinyon and Phillips \cite{phd115,phd121} in the present millennium. One of the most important results in the theory of Bol loops is the solution of the open problem on the existence of a simple Bol loop which was finally laid to rest by Nagy \cite{nagy1,nagy2,nagy3}.

In 1978, Sharma and Sabinin \cite{32,33} introduced and studied the algebraic properties of the notion of half-Bol loops(left B-loops). Thereafter, Adeniran \cite{phd111}, Adeniran and Akinleye \cite{1}, Adeniran and Solarin \cite{3} studied the algebraic properties of generalized Bol loops. Also, Ajmal \cite{5} introduced and studied the algebraic properties of generalized Bol loops and their relationship with M-loops.

Some of their results are highlighted below.
\begin{myth}(Adeniran and Akinleye \cite{1})\label{1}

If $(L,\cdot )$ is a generalized Bol
loop, then:
\begin{enumerate}
\item $(L,\cdot)$ is an RIPL.
\item $x^{\lambda}=x^{\rho}$ for all $x\in L$.
\item $R_{y\cdot \sigma(y)}=R_yR_{\sigma(y)}$ for all $y\in L$.
\item $[xy\cdot\sigma(x)]^{-1}=(\sigma(x))^{-1}y^{-1}\cdot x^{-1}$ for all $x,y\in L$.
\item $(R_{y^{-1}}, L_yR_{\sigma(y)},
R_{\sigma(y)}),(R_y^{-1}, L_yR_{\sigma(y)},
R_{\sigma(y)})\in AUT(L,\cdot )$ for all $y\in L$.
\end{enumerate}
\end{myth}
\begin{myth}(Sharma and Sabinin \cite{32})\label{2}

If $(L,\cdot )$ is a half Bol loop, then:
\begin{enumerate}
\item $(L,\cdot)$ is an LIPL.
\item $x^{\lambda}=x^{\rho}$ for all $x\in L$.
\item $L_{(x)}L_{(\sigma(x))} =
L_{(\sigma(x)x)}$ for all $x\in L$.
\item $(\sigma(x)\cdot yx)^{-1} = x^{-1}\cdot y^{-1} (\sigma(x))^{-1}$ for all $x,y\in L$.
\item $(R_{(x)}L_{(\sigma(x))},
L_{(x)^{-1}}, L_{(\sigma(x))}),(R_{(\sigma(x))}L_{(x)^{-1}}, L_{\sigma(x)}, L_{(x)^{-1}})\in AUT(L,\cdot )$ for all $x\in L$.
\end{enumerate}
\end{myth}
\begin{myth}(Ajmal \cite{5})\label{3}

Let $(L,\cdot )$ be a loop. The following statements are equivalent:
\begin{enumerate}
\item $(L,\cdot)$ is an M-loop;
\item $(L,\cdot)$ is both a left B-loop and a right B-loop;
\item $(L,\cdot)$ is a right B-loop and satisfies the LIP;
\item $(L,\cdot)$ is a left B-loop and satisfies the RIP.
\end{enumerate}
\end{myth}
\begin{myth}(Ajmal \cite{5})\label{4}

Every isotope of a right B-loop with the LIP is a right B-loop.
\end{myth}

\begin{myexam}
Let $R$ be a ring of all $2\times 2$ matrices taken over the field of three elements and let $G=R\times R$. For all $(u,f),(v,g)\in G$,
define $(u,f)\cdot (v,g)=(u+v,f+g+uv^3)$. Then $(G,\cdot )$ is a loop which is not a right Bol loop but which is a $\sigma$-generalized Bol loop with $\sigma~:x\mapsto x^2$ .
\end{myexam}
We shall need the following result.
\begin{myth}\label{sim}(Belousov \cite{belousov})

Let $(G,\cdot )$ be a loop with an identity element $e$. Let
\begin{displaymath}
\psi : \mathcal{P}(G,\cdot)\rightarrow N_{\rho}(G,\cdot)\uparrow \psi(U)=eU,~\phi :\Phi(G,\cdot)\rightarrow \Psi(G,\cdot)~\uparrow \phi(U)=U',
\end{displaymath}
\begin{displaymath}
\varpi :\Phi(G,\cdot)\rightarrow N_{\mu}(G,\cdot)\uparrow \varpi(U)=eU~ \textrm{and}~ \beta: \Psi(G,\cdot)\rightarrow N_{\mu}(G,\cdot)\uparrow \beta(U')=eU'.
\end{displaymath}
Then $\mathcal{P}(G,\cdot )\stackrel{\psi}{\cong} N_{\rho}(G,\cdot ),~\Phi(G,\cdot )\stackrel{\phi}{\cong}\Psi(G,\cdot ),~\Phi(G,\cdot )\stackrel{\varpi}{\cong}N_\mu(G,\cdot ),~\Psi(G,\cdot )\stackrel{\beta}{\cong}N_\mu(G,\cdot )$.
\end{myth}

\paragraph{}
Interestingly, Adeniran \cite{phd79} and
Robinson \cite{25}, Chiboka and Solarin
\cite{phd80}, Bruck \cite{7}, Bruck and Paige \cite{phd40},
Robinson \cite{phd7}, Huthnance \cite{phd44} and  Adeniran \cite{phd79} have respectively studied the holomorphs of Bol loops,
conjugacy closed loops, inverse property loops,
A-loops, extra loops, weak inverse property loops and Bruck loops. A set of results on the holomorph of some varieties of loops can be found in Jaiyeola \cite{davidref:10}. The latest study on the holomorph of generalized Bol loops can be found in Adeniran et. al. \cite{ho1}.

In this present work, the notion of the holomorph of a generalized Bol loop (GBL) is characterized afresh. The holomorph of a right inverse property loop (RIPL) is shown to be a GBL if and only if the loop is a GBL and some bijections of the loop are right (middle) regular. The holomorph of a RIPL is shown to be a GBL if and only if the loop is a GBL and some elements of the loop are right (middle) nuclear. Necessary and sufficient condition for the holomorph of a RIPL to be a Bol loop are deduced. Some algebraic properties and commutative diagrams are established for a RIPL whose holomorph is a GBL.

\section{Main Results}
\begin{myth}\label{4.1}
Let $(Q,\cdot)$ be a RIPL with a self map $\sigma$ and let $(H,\circ )$ be the A-holomorph of $(Q,\cdot)$ with a self map $\sigma'$ such that $\sigma'~:~(\alpha,x)\mapsto (\alpha,\sigma(x))$ for all $(\alpha,x)\in H$. The A-holomorph $(H,\circ)$ of $(Q,\cdot)$ is a $\sigma'$-generalised
Bol loop if and only if $C=\Big(R_{x}^{-1}, L_xR_{[\sigma(x\gamma^{-1})]\alpha^{-1}},R_{[\sigma(x\gamma^{-1})]\alpha^{-1}}\Big)\in AUT(Q,\cdot)$
for all $x\in Q$ and all $\alpha,\gamma\in A(Q)$.
\end{myth}
\begin{proof}
Note that \begin{itemize}
\item $(H,\circ)$ is a RIPL if and only if $(Q,\cdot)$ is a RIPL.
\item $(Q,\cdot)$ is a $\sigma$-generalised Bol loop if and only if
$B=(R_{x}^{-1}, L_xR_{\sigma(x)}, R_{\sigma(x)})\in AUT(Q,\cdot)$ for all $x\in Q$.
\end{itemize}
Define $\sigma'~:H\to H$ as $\sigma'(\alpha,x)=(\alpha,\sigma(x))$. Let $(\alpha,x),(\beta,y),(\gamma,z)\in H$, then
$(H,\circ)$ is a $\sigma'$-generalised Bol loop if and only if
$(\mathbb{R}_{(\alpha,x)^{-1}}, \mathbb{L}_{(\alpha,x)}\mathbb{R}_{\sigma'(\alpha,x)}, \mathbb{R}_{\sigma'(\alpha,x)})\in AUT(H,\circ)$ for all $(\alpha,x)\in H$, i.e. $(\mathbb{R}_{(\alpha,x)^{-1}}, \mathbb{L}_{(\alpha,x)}\mathbb{R}_{(\alpha,\sigma (x))}, \mathbb{R}_{(\alpha,\sigma (x))})\in AUT(H,\circ )\Longleftrightarrow$
\begin{gather}
(\beta,y)\mathbb{R}_{(\alpha,x)^{-1}}\circ(\gamma,z)\mathbb{L}_{(\alpha,x)}\mathbb{R}_{(\alpha,\sigma (x))}
= [(\beta,y)
\circ(\gamma,z)]\mathbb{R}_{(\alpha,\sigma (x))}\nonumber\\
\Leftrightarrow[(\beta,y)\circ(\alpha,x)^{-1}]\circ [((\alpha,x)\circ(\gamma,z))\circ(\alpha,\sigma(x))]
=[(\beta,y)\circ(\gamma,z)]\circ
(\alpha,\sigma(x))\label{eq:25}
\end{gather}
Let $(\beta,y)\circ (\alpha,x)^{-1}=(\tau,t)$. Since $(\alpha,x)^{-1}=(\alpha^{-1},(x\alpha^{-1})^{-1})$, then
\begin{equation}\label{eq:26}
(\tau,t)=(\beta\alpha^{-1},(yx^{-1})\alpha^{-1})
\end{equation}
From \eqref{eq:25} and \eqref{eq:26},
\begin{gather}
(\tau,t)\circ[(\alpha\gamma,x\gamma\cdot
z)\circ(\alpha,\sigma(x))]=(\beta
\gamma,y\gamma\cdot z)\circ(\alpha,\sigma(x))\nonumber\\
\Leftrightarrow\Big(\tau\alpha\gamma\alpha,(t\alpha\gamma\alpha)\big((x\gamma\cdot
z)\alpha\cdot\sigma(x)\big)\Big)
=\big(\beta\gamma\alpha,(y\gamma\cdot z)\alpha\cdot\sigma(x)\big)\label{eq:28}
\end{gather}
Putting \eqref{eq:26} in \eqref{eq:28}, we have
\begin{gather}
\Big(\beta\alpha^{-1}\alpha\gamma\alpha,(yx^{-1})\alpha^{-1}(\alpha\gamma\alpha)\big((x\gamma\cdot
z)\alpha\cdot\sigma(x)\big)\Big) = \big(\beta\gamma\alpha,(y\gamma\cdot
z)\alpha\cdot\sigma(x)\big)\nonumber\\
\Leftrightarrow\Big(\beta\gamma\alpha, (yx^{-1})\gamma\alpha\big[(x\gamma\cdot z)\alpha\cdot\sigma(x)\big]\Big)
=\big(\beta\gamma\alpha,(y\gamma\cdot z)\alpha\cdot\sigma(x)\big)\nonumber\\
\Leftrightarrow(yx^{-1})\gamma\alpha\cdot [(x\gamma\cdot z)\alpha\cdot\sigma(x)]=(y\gamma\cdot z)\alpha\cdot\sigma(x)\nonumber\\
\Leftrightarrow \big[(yx^{-1})\gamma\cdot [(x\gamma\cdot z)\cdot
(\sigma(x)\alpha^{-1})]\big]\alpha=[(y\gamma\cdot z)
\cdot(\sigma(x)\alpha^{-1})]\alpha\nonumber\\
\Leftrightarrow (y\gamma x^{-1}\gamma )[(x\gamma\cdot z)\cdot (\sigma(x)\alpha^{-1})]=(y\gamma\cdot z)(\sigma(x)\alpha^{-1})\label{eq:33}
\end{gather}
Let $\bar{y}=y\gamma$, then \eqref{eq:33} becomes
\begin{gather}
(\bar{y}\cdot x^{-1}\gamma)[(x\gamma\cdot z)(\sigma(x)\alpha^{-1})]=(\bar{y}\cdot z)(\sigma(x)\alpha^{-1})\nonumber\\
\Leftrightarrow\Big(R_{x\gamma}^{-1}, L_{x\gamma}R_{[\sigma(x)\alpha^{-1}]},R_{[\sigma(x)\alpha^{-1}]}\Big)\in AUT(Q,\cdot)\nonumber\\
\textrm{and replacing $x\gamma$ by $x$,}~\Big(R_{x}^{-1}, L_xR_{[\sigma(x\gamma^{-1})]\alpha^{-1}},R_{[\sigma(x\gamma^{-1})]\alpha^{-1}}\Big)\in AUT(Q,\cdot)\nonumber .\qed
\end{gather}
\end{proof}

\begin{myth}\label{4.2}
Let $(Q,\cdot)$ be a RIPL with a self map $\sigma$
 and let $(H,\circ )$ be the A-holomorph of $(Q,\cdot)$ with a self map $\sigma'$ such that $\sigma'~:~(\alpha,x)\mapsto (\alpha,\sigma(x))$ for all $(\alpha,x)\in H$. $(H,\circ)$ is a $\sigma'$-generalised
Bol loop if and only if
\begin{enumerate}
  \item $(Q,\cdot)$ is a $\sigma$-GBL;
  \item $\Big(I,R_{\sigma (x)}^{-1}R_{\sigma(x\gamma^{-1})}, R_{\sigma (x)}^{-1}R_{\sigma(x\gamma^{-1})}\Big)\in AUT(Q,\cdot)$; and
  \item $\Big(I,R_{\sigma (x)}^{-1}R_{[\sigma(x)]\alpha^{-1}},R_{\sigma (x)}^{-1}R_{[\sigma(x)]\alpha^{-1}}\Big)\in AUT(Q,\cdot)$
\end{enumerate}
for all $x\in Q$ and $\alpha,\gamma\in A(Q)$.
\end{myth}
\begin{proof}
From Theorem~\ref{4.1}, $(H,\circ)$ is a $\sigma'$-generalised
Bol loop if and only if
$$C=\Big(R_{x}^{-1}, L_xR_{[\sigma(x\gamma^{-1})]\alpha^{-1}},R_{[\sigma(x\gamma^{-1})]\alpha^{-1}}\Big)\in AUT(Q,\cdot)\Leftrightarrow \Big(R_{x}^{-1}, L_xR_{\sigma''(x)},R_{\sigma''(x)}\Big)\in AUT(Q,\cdot)$$
where $\sigma''(x)=[\sigma(x\gamma^{-1})]\alpha^{-1}$. Taking $\alpha =\gamma =I$ in $C$, then $\sigma''=\sigma$ which implies that $(Q,\cdot)$ is a $\sigma$-GBL and thus $B=(R_{x}^{-1}, L_xR_{\sigma(x)}, R_{\sigma(x)})\in AUT(Q,\cdot)$ for all $x\in Q$. So,
\begin{equation}\label{gbl1}
B^{-1}C=\Big(I,R_{\sigma (x)}^{-1}R_{[\sigma(x\gamma^{-1})]\alpha^{-1}}, R_{\sigma (x)}^{-1}R_{[\sigma(x\gamma^{-1})]\alpha^{-1}}\Big)\in AUT(Q,\cdot)
\end{equation}
Substitute $\alpha =I$ in \eqref{gbl1} to get
\begin{equation*}
    D(x)=\Big(I,R_{\sigma (x)}^{-1}R_{[\sigma(x\gamma^{-1})]}, R_{\sigma (x)}^{-1}R_{[\sigma(x\gamma^{-1})]}\Big)\in AUT(Q,\cdot)
\end{equation*} and also substitute $\gamma =I$ in \eqref{gbl1} to get
\begin{equation*}
E(x)=\Big(I,R_{\sigma (x)}^{-1}R_{[\sigma(x)]\alpha^{-1}}, R_{\sigma (x)}^{-1}R_{[\sigma(x)]\alpha^{-1}}\Big)\in AUT(Q,\cdot)
\end{equation*}
This proves the forward. The converse is achieved by computing and showing that $BD(x)E(x\gamma^{-1})=C$.\qed
 \end{proof}

\begin{myth}\label{5}
Let $(Q,\cdot)$ be a RIPL with a self map $\sigma$
 and let $(H,\circ )$ be the A-holomorph of $(Q,\cdot)$ with a self map $\sigma'$ such that $\sigma'~:~(\alpha,x)\mapsto (\alpha,\sigma(x))$ for all $(\alpha,x)\in H$. $(H,\circ)$ is a $\sigma'$-generalised
Bol loop if and only if
\begin{enumerate}
  \item $(Q,\cdot)$ is a $\sigma$-GBL; and
    \item $\sigma (x)^{-1}\sigma\big(x\gamma^{-1}\big),\sigma (x)^{-1}(\sigma(x))\alpha^{-1}\in N_\rho(Q,\cdot)$;
\end{enumerate}
for all $x,y\in Q$ and $\alpha,\gamma\in A(Q)$.
\end{myth}
\begin{proof}
This is achieved by Theorem~\ref{4.2} by using the autotopisms $D(x)$ and $E(x)$.\qed
\end{proof}

\begin{mylem}\label{6}
Let $(Q,\cdot)$ be a RIPL with a bijective self map $\sigma$ and let $(H,\circ )$ be the holomorph of $(Q,\cdot)$ with a self map $\sigma'$ such that $\sigma'~:~(\alpha,x)\mapsto (\alpha,\sigma(x))$ for all $(\alpha,x)\in H$. If $(H,\circ)$ is a $\sigma'$-GBL, then $A(Q)=\{\sigma R_{n_1}\sigma^{-1},R_{n_2}^{-1}|n_1,n_2\in N_\rho(Q,\cdot)\}$.
\end{mylem}
\begin{proof}
Using Theorem~\ref{5}:
\begin{gather*}
\sigma (x)\cdot\sigma (x)^{-1}\sigma\big(x\gamma^{-1}\big),\sigma (x)\cdot\sigma (x)^{-1}(\sigma(x))\alpha^{-1}\in \sigma (x)N_\rho(Q,\cdot)\\
\Rightarrow\sigma\big(x\gamma^{-1}\big)=\sigma(x)n_1~\textrm{and}~(\sigma(x))\alpha^{-1}=\sigma(x)n_2~\textrm{for some}~n_1,n_2\in N_\rho(Q,\cdot)\\
\Rightarrow \gamma=\sigma R_{n_1}\sigma^{-1}~\textrm{and}~\alpha=R_{n_2}^{-1}~\textrm{for some}~n_1,n_2\in N_\rho(Q,\cdot).\qed
\end{gather*}
\end{proof}

\begin{myth}\label{7}
Let $(Q,\cdot)$ be a RIPL with a self map $\sigma$
 and let $(H,\circ )$ be the A-holomorph of $(Q,\cdot)$ with a self map $\sigma'$ such that $\sigma'~:~(\alpha,x)\mapsto (\alpha,\sigma(x))$ for all $(\alpha,x)\in H$. $(H,\circ)$ is a $\sigma'$-generalised
Bol loop if and only if
\begin{enumerate}
  \item $(Q,\cdot)$ is a $\sigma$-GBL;
  \item $\Big(R_{\sigma (x)}^{-1}R_{\sigma(x\gamma^{-1})}, \big(JR_{\sigma(x\gamma^{-1})}^{-1}R_{\sigma (x)}J\big)^{-1},I\Big)\in AUT(Q,\cdot)$; and
  \item $\Big(R_{\sigma (x)}^{-1}R_{[\sigma(x)]\alpha^{-1}},\big(JR_{[\sigma(x)]\alpha^{-1}}^{-1}R_{\sigma (x)}J\big)^{-1},I\Big)\in AUT(Q,\cdot)$
\end{enumerate}
for all $x\in Q$ and $\alpha,\gamma\in A(Q)$.
\end{myth}
\begin{proof}
This is achieved with Theorem~\ref{4.2} by using the fact that in a RIPL, $(U,V,W)\in AUT(Q,\cdot)\Rightarrow (W,JVJ,U)\in AUT(Q,\cdot)$.\qed
\end{proof}

\begin{myth}\label{7.1}
Let $(Q,\cdot)$ be a RIPL with a self map $\sigma$
 and let $(H,\circ )$ be the A-holomorph of $(Q,\cdot)$ with a self map $\sigma'$ such that $\sigma'~:~(\alpha,x)\mapsto (\alpha,\sigma(x))$ for all $(\alpha,x)\in H$. The following are equivalent
 \begin{enumerate}
 \item $(H,\circ)$ is a $\sigma'$-GBL.
\item
\begin{enumerate}
   \item $(Q,\cdot)$ is a $\sigma$-GBL;
  \item $R_{\sigma (x)}^{-1}R_{\sigma(x\gamma^{-1})}$ and $R_{\sigma (x)}^{-1}R_{[\sigma(x)]\alpha^{-1}}$ are $\rho$-regular for all $x\in Q$ and $\alpha,\gamma\in A(Q)$.
\end{enumerate}
\item
\begin{enumerate}
 \item $(Q,\cdot)$ is a $\sigma$-GBL;
  \item $R_{\sigma (x)}^{-1}R_{\sigma(x\gamma^{-1})}$ and $R_{\sigma (x)}^{-1}R_{[\sigma(x)]\alpha^{-1}}$ are $\mu$-regular with adjoints $JR_{\sigma(x\gamma^{-1})}^{-1}R_{\sigma (x)}J$ and $JR_{[\sigma(x)]\alpha^{-1}}^{-1}R_{\sigma (x)}J$ respectively, for all $x\in Q$ and $\alpha,\gamma\in A(Q)$.
\end{enumerate}
 \end{enumerate}
\end{myth}
\begin{proof}
Use Theorem~\ref{4.2} and Theorem~\ref{7}.\qed
\end{proof}

\begin{mycor}\label{8}
Let $(Q,\cdot)$ be a RIPL with a self map $\sigma$
 and let $(H,\circ )$ be the A-holomorph of $(Q,\cdot)$ with a self map $\sigma'$ such that $\sigma'~:~(\alpha,x)\mapsto (\alpha,\sigma(x))$ for all $(\alpha,x)\in H$. The following are equivalent
 \begin{enumerate}
 \item $(H,\circ)$ is a $\sigma'$-GBL.
\item
\begin{enumerate}
   \item $(Q,\cdot)$ is a $\sigma$-GBL;
  \item $R_{\sigma (x)}^{-1}R_{\sigma(x\gamma^{-1})},R_{\sigma (x)}^{-1}R_{[\sigma(x)]\alpha^{-1}}\in \mathcal{P}(Q,\cdot)$ for all $x\in Q$ and $\alpha,\gamma\in A(Q)$.
\end{enumerate}
\item
\begin{enumerate}
 \item $(Q,\cdot)$ is a $\sigma$-GBL;
  \item $R_{\sigma (x)}^{-1}R_{\sigma(x\gamma^{-1})},R_{\sigma (x)}^{-1}R_{[\sigma(x)]\alpha^{-1}}\in \Phi (Q,\cdot)$ and  $JR_{\sigma(x\gamma^{-1})}^{-1}R_{\sigma (x)}J,JR_{[\sigma(x)]\alpha^{-1}}^{-1}R_{\sigma (x)}J\in\Psi(Q,\cdot)$ for all $x\in Q$ and $\alpha,\gamma\in A(Q)$.
\end{enumerate}
 \end{enumerate}
\end{mycor}
\begin{proof}
Use Theorem~\ref{7.1}.\qed
\end{proof}

\begin{mycor}\label{0.9}
Let $(Q,\cdot)$ be a RIPL with a self map $\sigma$
 and let $(H,\circ )$ be the A-holomorph of $(Q,\cdot)$ with a self map $\sigma'$ such that $\sigma'~:~(\alpha,x)\mapsto (\alpha,\sigma(x))$ for all $(\alpha,x)\in H$. The following are equivalent
 \begin{enumerate}
 \item $(H,\circ)$ is a $\sigma'$-GBL.
\item
\begin{enumerate}
   \item $(Q,\cdot)$ is a $\sigma$-GBL;
  \item $R_{\sigma(x\gamma^{-1})},R_{[\sigma(x)]\alpha^{-1}}\in R_{\sigma (x)}\mathcal{P}(Q,\cdot)$ for all $x\in Q$ and $\alpha,\gamma\in A(Q)$.
\end{enumerate}
\item
\begin{enumerate}
 \item $(Q,\cdot)$ is a $\sigma$-GBL;
  \item $R_{\sigma(x\gamma^{-1})},R_{[\sigma(x)]\alpha^{-1}}\in R_{\sigma (x)}\Phi (Q,\cdot)$ and  $R_{\sigma (x)}J\in R_{\sigma(x\gamma^{-1})}J\Psi(Q,\cdot),R_{\sigma (x)}J\in R_{[\sigma(x)]\alpha^{-1}}J\Psi(Q,\cdot)$ for all $x\in Q$ and $\alpha,\gamma\in A(Q)$.
\end{enumerate}
 \end{enumerate}
\end{mycor}
\begin{proof}
Use Corollary~\ref{8}.\qed
\end{proof}

\begin{mylem}\label{8.3}
Let $(L,\cdot)$ be a loop. Then
\begin{enumerate}
\item $\delta\mathcal{P}(L,\cdot )\delta^{-1}=\mathcal{P}(L,\cdot )$ for all $\delta\in AUM(L,\cdot)$.
\item $\delta\Phi(L,\cdot )\delta^{-1}=\Phi(L,\cdot )$ and $\delta\Psi(L,\cdot )\delta^{-1}=\Psi(L,\cdot )$ for all $\delta\in AUM(L,\cdot)$.
\end{enumerate}
\end{mylem}
\begin{proof}
\begin{enumerate}
\item Let $\delta\in AUM(L,\cdot)$ and $U\in\mathcal{P}(L,\cdot )$. Then $(\delta,\delta,\delta)(I,U,U)(\delta^{-1},\delta^{-1},\delta^{-1})=
(I,\delta U\delta^{-1},\delta U\delta^{-1})\in AUT(L,\cdot)\Rightarrow \delta U\delta^{-1}\in\mathcal{P}(L,\cdot )$. Hence the conclusion.
\item These are similar to the proof of 1. \qed
\end{enumerate}
\end{proof}

\begin{mycor}\label{8.4}
Let $(Q,\cdot)$ be a RIPL with a self map $\sigma$
 and let $(H,\circ )$ be the A-holomorph of $(Q,\cdot)$ with a self map $\sigma'$ such that $\sigma'~:~(\alpha,x)\mapsto (\alpha,\sigma(x))$ for all $(\alpha,x)\in H$. If $(H,\circ)$ is a $\sigma'$-GBL, then
 \begin{enumerate}
\item $$\delta R_{\sigma (x)}^{-1}R_{\sigma(x\gamma^{-1})}\delta^{-1},\delta R_{\sigma (x)}^{-1}R_{[\sigma(x)]\alpha^{-1}} \delta^{-1}\in\mathcal{P}(L,\cdot )~\textrm{for all}~\delta\in AUM(L,\cdot).$$
     In particular, $\alpha R_{\sigma (x)}^{-1}R_{\sigma(x\gamma^{-1})}\alpha^{-1},\gamma R_{\sigma (x)}^{-1}R_{[\sigma(x)]\alpha^{-1}} \gamma^{-1}\in\mathcal{P}(L,\cdot )~\textrm{for all}~x\in L.$
\item $\delta R_{\sigma (x)}^{-1}R_{\sigma(x\gamma^{-1})}\delta^{-1},\delta R_{\sigma (x)}^{-1}R_{[\sigma(x)]\alpha^{-1}}\delta^{-1}\in \Phi(L,\cdot )$ and
    $$\delta J R_{\sigma(x\gamma^{-1})}^{-1}R_{\sigma (x)}(\delta J)^{-1},\delta R_{[\sigma(x)]\alpha^{-1}}^{-1}R_{\sigma (x)} (\delta J)^{-1}\in\Psi(L,\cdot )~\textrm{for all}~\delta\in AUM(L,\cdot).$$
     In particular, $\alpha R_{\sigma (x)}^{-1}R_{\sigma(x\gamma^{-1})}\alpha^{-1},\gamma R_{\sigma (x)}^{-1}R_{[\sigma(x)]\alpha^{-1}}\gamma^{-1}\in \Phi(L,\cdot )$

      and $\alpha J R_{\sigma(x\gamma^{-1})}^{-1}R_{\sigma (x)}(\alpha J)^{-1},\gamma JR_{[\sigma(x)]\alpha^{-1}}^{-1}R_{\sigma (x)} (\gamma J)^{-1}\in\Psi(L,\cdot )$ for all $x\in L$.
\end{enumerate}
\end{mycor}
\begin{proof}
Use Corollary~\ref{8} and Lemma~\ref{8.3}.\qed
\end{proof}

\begin{mycor}\label{9}
Let $(Q,\cdot)$ be a RIPL with a self map $\sigma$
 and let $(H,\circ )$ be the A-holomorph of $(Q,\cdot)$ with a self map $\sigma'$ such that $\sigma'~:~(\alpha,x)\mapsto (\alpha,\sigma(x))$ for all $(\alpha,x)\in H$. The following are equivalent
 \begin{enumerate}
 \item $(H,\circ)$ is a $\sigma'$-GBL.
\item
\begin{enumerate}
   \item $(Q,\cdot)$ is a $\sigma$-GBL;
  \item $\sigma (x)^{-1}\sigma(x\gamma^{-1}),\sigma (x)^{-1}[\sigma(x)]\alpha^{-1}\in N_\rho (Q,\cdot)$ for all $x\in Q$ and $\alpha,\gamma\in A(Q)$.
\end{enumerate}
\item
\begin{enumerate}
 \item $(Q,\cdot)$ is a $\sigma$-GBL;
  \item $\sigma (x)^{-1}\sigma(x\gamma^{-1}),\sigma (x)^{-1}[\sigma(x)]\alpha^{-1}\in N_\mu(Q,\cdot),\big(\sigma(x\gamma^{-1})\big)^{-1}\sigma (x),\big([\sigma(x)]\alpha^{-1})^{-1}\sigma (x)\in N_\mu(Q,\cdot)$ for all $x\in Q$ and $\alpha,\gamma\in A(Q)$.
\end{enumerate}
 \end{enumerate}
\end{mycor}
\begin{proof}
We shall use Corollary~\ref{8} and Theorem~\ref{sim}.
\begin{enumerate}
  \item Since $\mathcal{P}(G,\cdot )\stackrel{\psi}{\cong} N_{\rho}(G,\cdot )$, then $R_{\sigma (x)}^{-1}R_{\sigma(x\gamma^{-1})},R_{\sigma (x)}^{-1}R_{[\sigma(x)]\alpha^{-1}}\in \mathcal{P}(Q,\cdot)\Leftrightarrow eR_{\sigma (x)}^{-1}R_{\sigma(x\gamma^{-1})},eR_{\sigma (x)}^{-1}R_{[\sigma(x)]\alpha^{-1}}\in N_{\rho}(G,\cdot )\Leftrightarrow\sigma (x)^{-1}\sigma(x\gamma^{-1}),\sigma (x)^{-1}[\sigma(x)]\alpha^{-1}\in N_\rho (Q,\cdot)$ for all $x\in Q$ and $\alpha,\gamma\in A(Q)$.
  \item This is similar to 1.\qed
\end{enumerate}
\end{proof}

\begin{myth}\label{10}
Let $(Q,\cdot)$ be a RIPL with a self map $\sigma$
 and let $(H,\circ )$ be the A-holomorph of $(Q,\cdot)$ with a self map $\sigma'$ such that $\sigma'~:~(\alpha,x)\mapsto (\alpha,\sigma(x))$ for all $(\alpha,x)\in H$. If $(H,\circ)$ is a $\sigma'$-GBL, then:
 \begin{enumerate}
   \item $JR_{\sigma(x\gamma^{-1})}^{-1}R_{\sigma (x)}J=L_{\sigma (x)^{-1}\sigma(x\gamma^{-1})}$;
   \begin{enumerate}
     \item $\Big[y^{-1}[\sigma(x\gamma^{-1})]^{-1}\cdot \sigma (x)]^{-1}=\big[\sigma (x)^{-1}\sigma(x\gamma^{-1})\big]y$,
     \item $\big[\sigma(x\gamma^{-1})^{-1}\sigma (x)\big]^{-1}=\sigma (x)^{-1}\sigma(x\gamma^{-1})$.
   \end{enumerate}
   \item $R_{\sigma (x)}^{-1}R_{\sigma(x\gamma^{-1})}=R_{\Big\{[\sigma(x\gamma^{-1})]^{-1}\sigma (x)\Big\}^{-1}}=R_{\sigma (x)^{-1}\sigma(x\gamma^{-1})}$;
       \begin{enumerate}
         \item $y\sigma (x)^{-1}\cdot \sigma(x\gamma^{-1})=y\Big\{[\sigma(x\gamma^{-1})]^{-1}\sigma (x)^{-1}\Big\}^{-1}=y\big[\sigma          (x)^{-1}\cdot \sigma(x\gamma^{-1})\big]$,
         \item $\sigma (x)\Big\{[\sigma(x\gamma^{-1})]^{-1}\sigma (x)^{-1}\Big\}^{-1}=\sigma(x\gamma^{-1})$.
       \end{enumerate}
   \item $JR_{[\sigma(x)]\alpha^{-1}}^{-1}R_{\sigma (x)}J=L_{\sigma (x)^{-1}[\sigma(x)]\alpha^{-1}}$;
   \begin{enumerate}
     \item $y^{-1}\big(([\sigma(x)]\alpha^{-1})^{-1}\cdot\sigma (x)\big)^{-1}=\big[\sigma (x)^{-1}[(\sigma(x))\alpha^{-1}]\big]y$;
     \item $\big(([\sigma(x)]\alpha^{-1})^{-1}\sigma (x)\big)^{-1}=\sigma (x)^{-1}\big[(\sigma(x))\alpha^{-1}\big]$.
   \end{enumerate}
   \item $R_{\sigma (x)}^{-1}R_{[\sigma(x)]\alpha^{-1}}=R_{\Big\{\big[(\sigma(x))\alpha^{-1}\big]^{-1}\sigma (x)\Big\}^{-1}}=R_{\sigma (x)^{-1}}R_{[\sigma(x)]\alpha^{-1}}$;
       \begin{enumerate}
         \item $y(\sigma (x))^{-1}\cdot [\sigma(x)]\alpha^{-1}=y\Big\{\big[(\sigma(x))\alpha^{-1}\big]^{-1}\sigma (x)\Big\}^{-1}=y\big[\sigma (x)^{-1}\cdot [(\sigma(x))\alpha^{-1}]\big]$,
         \item $\sigma (x)\Big\{\big[(\sigma(x))\alpha^{-1}\big]^{-1}\sigma (x)\Big\}^{-1}=(\sigma (x))\alpha^{-1}$.
       \end{enumerate}
 \end{enumerate}
\end{myth}
\begin{proof}
\begin{enumerate}
  \item From Theorem~\ref{7}, $\Big(R_{\sigma (x)}^{-1}R_{\sigma(x\gamma^{-1})}, \big(JR_{\sigma(x\gamma^{-1})}^{-1}R_{\sigma (x)}J\big)^{-1},I\Big)\in AUT(Q,\cdot)$ implies
      \begin{equation*}
        yR_{\sigma (x)}^{-1}R_{\sigma(x\gamma^{-1})}\cdot z=y\cdot zJR_{\sigma(x\gamma^{-1})}^{-1}R_{\sigma (x)}J.
      \end{equation*}
              Put $y=e$ to get $JR_{\sigma(x\gamma^{-1})}^{-1}R_{\sigma (x)}J=L_{\sigma (x)^{-1}\sigma(x\gamma^{-1})}$. (a) and (b) follow from this.
  \item From Theorem~\ref{4.2}, $\Big(I,R_{\sigma (x)}^{-1}R_{\sigma(x\gamma^{-1})}, R_{\sigma (x)}^{-1}R_{\sigma(x\gamma^{-1})}\Big)\in AUT(Q,\cdot)$ implies
       \begin{equation*}
        y\cdot zR_{\sigma (x)}^{-1}R_{\sigma(x\gamma^{-1})}=(yz)R_{\sigma (x)}^{-1}R_{\sigma(x\gamma^{-1})}.
       \end{equation*}
               Put $z=e$ and subsequently $y=e$ to get

               $R_{\sigma (x)}^{-1}R_{\sigma(x\gamma^{-1})}=R_{\Big\{[\sigma(x\gamma^{-1})]^{-1}\sigma (x)\Big\}^{-1}}=R_{\sigma (x)^{-1}\sigma(x\gamma^{-1})}$. (a) and (b) follow from this.
  \item This is similar to 1.
  \item This is similar to 2.\qed

\end{enumerate}
\end{proof}

\begin{myth}\label{11}
Let $(Q,\cdot)$ be a RIPL with a bijective self map $\sigma$
 and let $(H,\circ )$ be the A-holomorph of $(Q,\cdot)$ with a self map $\sigma'$ such that $\sigma'~:~(\alpha,x)\mapsto (\alpha,\sigma(x))$ for all $(\alpha,x)\in H$. The following are equivalent
 \begin{enumerate}
 \item $(H,\circ)$ is a $\sigma'$-GBL.
\item
\begin{enumerate}
   \item $(Q,\cdot)$ is a $\sigma$-GBL;
  \item $\sigma (x)^{-1}\sigma^2\big(\sigma^{-1}(x)\cdot n\big)\in N_\rho (Q,\cdot)~\ni~\gamma=\sigma R_{n}\sigma^{-1}~\forall~\gamma\in A(Q),~x\in Q$ and some $n\in N_\rho (Q,\cdot)$.
\end{enumerate}
\item
\begin{enumerate}
 \item $(Q,\cdot)$ is a $\sigma$-GBL;
 \item $\sigma (x)^{-1}\sigma^2\big(\sigma^{-1}(x)\cdot n\big)\in N_\mu (Q,\cdot)~\ni~\gamma=\sigma R_{n}\sigma^{-1}~\forall~\gamma\in A(Q)~,x\in Q$ and some $n\in N_\mu (Q,\cdot)$.
 \item $\big[\sigma^2\big(\sigma^{-1}(x)\cdot n\big)\big]^{-1}\sigma (x)\in N_\mu (Q,\cdot)~\ni~\gamma=\sigma R_{n}\sigma^{-1}~\forall~\gamma\in A(Q)~,x\in Q$ and some $n\in N_\mu (Q,\cdot)$.
  \item $(\sigma(x)\cdot n')^{-1}\sigma (x)\in N_\mu (Q,\cdot)~\ni~\alpha=R_{n'}^{-1}~\forall~\alpha\in A(Q)~,x\in Q$ and some $n'\in N_\mu (Q,\cdot)$.
 \end{enumerate}
  \end{enumerate}
  Hence, $\sigma (xn^{-1})=\sigma(x)n'^{-1}$ for all $x\in Q$ and some $n,n'\in N_\mu (Q,\cdot)$.
\end{myth}
\begin{proof}
This is achieved by Corollary~\ref{9} and Lemma~\ref{6}.\qed
\end{proof}

\begin{mycor}\label{12}
Let $(Q,\cdot)$ be a RIPL with a bijective self map $\sigma$
 and let $(H,\circ )$ be the A-holomorph of $(Q,\cdot)$ with a self map $\sigma'$ such that $\sigma'~:~(\alpha,x)\mapsto (\alpha,\sigma(x))$ for all $(\alpha,x)\in H$. $(H,\circ)$ is a $\sigma'$-GBL implies
 \begin{enumerate}
   \item $(Q,\cdot)$ is a $\sigma$-GBL.
  \item $\sigma (x)^{-1}\sigma^2\big(\sigma^{-1}(x)\cdot n\big)\in N_\rho (Q,\cdot)~\forall~x\in Q$ and some $n\in N_\rho (Q,\cdot)$.
  \item $\big[\sigma^2\big(\sigma^{-1}(x)\cdot n\big)\big]^{-1}\sigma (x)\in N_\mu (Q,\cdot)~\forall~x\in Q$ and some $n\in N_\mu (Q,\cdot)$.
 \item $(\sigma(x)\cdot n)^{-1}\sigma (x)\in N_\mu (Q,\cdot)~\forall~x\in Q$ and some $n\in N_\mu (Q,\cdot)$.
  \item $\sigma (xn^{-1})=\sigma(x)n'^{-1}$ for all $x\in Q$ and some $n,n'\in N_\mu (Q,\cdot)$.
\end{enumerate}
 \end{mycor}
\begin{proof}
This follows from Theorem~\ref{11}.\qed
\end{proof}

\begin{mycor}\label{13}
Let $(Q,\cdot)$ be a RIPL and let $(H,\circ )$ be the A-holomorph of $(Q,\cdot)$. The following are equivalent
 \begin{enumerate}
 \item $(H,\circ)$ is a Bol loop.
\item
\begin{enumerate}
   \item $(Q,\cdot)$ is a Bol loop;
  \item $\gamma=R_{n}^{-1}~\forall~\gamma\in A(Q)$ and some $n\in N_\rho (Q,\cdot)$.
\end{enumerate}
\item
\begin{enumerate}
 \item $(Q,\cdot)$ is a Bol loop;
 \item $\gamma=R_{n}^{-1}~\forall~\gamma\in A(Q)~,x\in Q$ and some $n\in N_\mu (Q,\cdot)$;
 \item $(x\cdot n)^{-1}x\in N_\mu (Q,\cdot)~\ni~\gamma=R_{n}^{-1}~\forall~\gamma\in A(Q)~,x\in Q$ and some $n\in N_\mu (Q,\cdot)$.
  \end{enumerate}
  \end{enumerate}
 \end{mycor}
\begin{proof}
This is achieved by Corollary~\ref{12} with $\sigma=I$.\qed
\end{proof}

\begin{myth}\label{14}
Let $(Q,\cdot)$ be a RIPL with a bijective self map $\sigma$
 and let $(H,\circ )$ be the A-holomorph of $(Q,\cdot)$ with a self map $\sigma'$ such that $\sigma'~:~(\alpha,x)\mapsto (\alpha,\sigma(x))$ for all $(\alpha,x)\in H$. The following are equivalent
 \begin{enumerate}
 \item $(H,\circ)$ is a $\sigma'$-GBL.
\item
\begin{enumerate}
   \item $(Q,\cdot)$ is a $\sigma$-GBL;
  \item $\gamma=\sigma \rho\sigma^{-1}$ for some $\rho\in\mathcal{P}(Q,\cdot)$ for all $\gamma\in A(Q)$;
   \item $\alpha\in\mathcal{P}(Q,\cdot)$ for all $\alpha\in A(Q)$.
\end{enumerate}
\item
\begin{enumerate}
\item $(Q,\cdot)$ is a $\sigma$-GBL;
  \item $\gamma=\sigma J\varphi(\sigma J)^{-1}$ and $\alpha=J\varphi J$ for some $\varphi\in\Psi(Q,\cdot)$ and for all $\gamma,\alpha\in A(Q)$;
   \item $\alpha =J\varphi J$ for some $\varphi\in\Psi(Q,\cdot)$ and for all $\alpha\in A(Q,\cdot)$.
 \end{enumerate}
 \end{enumerate}
\end{myth}
\begin{proof}
We need Corollary~\ref{0.9}.
\begin{gather*}
R_{\sigma(x\gamma^{-1})}\in R_{\sigma (x)}\mathcal{P}(Q,\cdot)\Leftrightarrow R_{\sigma(x\gamma^{-1})}=R_{\sigma (x)}\rho~\textrm{for some}~\rho\in\mathcal{P}(Q,\cdot)\Leftrightarrow\\
y\cdot \sigma(x\gamma^{-1})=(y\sigma (x))\rho\Leftrightarrow (I,\sigma^{-1}\gamma^{-1}\sigma,\rho)\in AUT(Q,\cdot)\Leftrightarrow\sigma^{-1}\gamma^{-1}\sigma=\rho\Leftrightarrow\\
\gamma=\sigma\rho^{-1}\sigma^{-1}
\Leftrightarrow\gamma=\sigma\rho_1\sigma^{-1}~\textrm{for some}~\rho_1\in\mathcal{P}(Q,\cdot).
\end{gather*}
\begin{gather*}
R_{[\sigma(x)]\alpha^{-1}}\in R_{\sigma (x)}\mathcal{P}(Q,\cdot)\Leftrightarrow R_{[\sigma(x)]\alpha^{-1}}=R_{\sigma (x)}\rho~\textrm{for some}~\rho\in\mathcal{P}(Q,\cdot)\Leftrightarrow\\
y\cdot [\sigma(x)]\alpha^{-1}=(y\sigma (x))\rho\Leftrightarrow (I,\alpha^{-1},\rho)\in AUT(Q,\cdot)\Leftrightarrow\alpha=\rho^{-1}\Leftrightarrow\\
\alpha=\rho_1~\textrm{for some}~\rho_1\in\mathcal{P}(Q,\cdot).
\end{gather*}
\begin{gather*}
R_{\sigma(x\gamma^{-1})}\in R_{\sigma (x)}\Phi (Q,\cdot)\Leftrightarrow R_{\sigma(x\gamma^{-1})}= R_{\sigma (x)}\varrho~\textrm{for some}~\varrho\in\Phi (Q,\cdot)\Leftrightarrow \\y\cdot \sigma(x\gamma^{-1})=(y\cdot \sigma (x))\varrho
\Leftrightarrow
(I,\sigma^{-1}\gamma^{-1}\sigma ,\varrho)\in AUT(Q,\cdot)\Leftrightarrow
(\varrho,J\sigma^{-1}\gamma^{-1}\sigma J,I)\in AUT(Q,\cdot) \\ \Leftrightarrow
 (\varrho,(J\sigma^{-1}\gamma\sigma J)^{-1},I)\in AUT(Q,\cdot)\Leftrightarrow\varrho'=J\sigma^{-1}\gamma\sigma J\Leftrightarrow\gamma=\sigma J\varrho'\gamma(\sigma J)^{-1}\Leftrightarrow\\
 \gamma=\sigma J\varphi\gamma(\sigma J)^{-1}~\textrm{for some}~\varphi\in\Psi (Q,\cdot).
\end{gather*}
\begin{gather*}
R_{\sigma (x)}J\in R_{\sigma(x\gamma^{-1})}J\Psi(Q,\cdot)\Leftrightarrow R_{\sigma (x)}J=R_{\sigma(x\gamma^{-1})}J\varphi~\textrm{for some}~\varphi\in\Psi (Q,\cdot)\Leftrightarrow \\
(y\cdot \sigma (x))^{-1}=\big([y\cdot \sigma(x\gamma^{-1})]^{-1}\big)\varphi\Leftrightarrow
(I,\sigma^{-1}\gamma\sigma, J\varphi J)\in AUT(Q,\cdot)\Leftrightarrow \\
(J\varphi J,J\sigma^{-1}\gamma\sigma J,I)\in AUT(Q,\cdot)\Leftrightarrow
\big(J\varphi J,(J\sigma^{-1}\gamma^{-1}\sigma J)^{-1},I\big)\in AUT(Q,\cdot)\Leftrightarrow \\
J\varphi J\in\Phi (Q,\cdot)~\textrm{and}~J\sigma^{-1}\gamma^{-1}\sigma J\in\Psi (Q,\cdot)\Leftrightarrow \varrho =J\varphi J\in\Phi (Q,\cdot)\\
~\textrm{for some}~\varrho\in\Phi (Q,\cdot)~\textrm{and}~\gamma=\sigma J\varphi(\sigma J)^{-1}~\textrm{for some}~\varphi\in\Psi (Q,\cdot)
\end{gather*}
\begin{gather*}
R_{[\sigma(x)]\alpha^{-1}}\in R_{\sigma (x)}\Phi (Q,\cdot)\Leftrightarrow R_{[\sigma(x)]\alpha^{-1}}= R_{\sigma (x)}\varrho~\textrm{for some}~\varrho\in\Phi (Q,\cdot)\Leftrightarrow \\y\cdot [\sigma(x)]\alpha^{-1}=(y\cdot \sigma (x))\varrho
\Leftrightarrow
(I,\alpha^{-1} ,\varrho)\in AUT(Q,\cdot)\Leftrightarrow
(\varrho,J\alpha^{-1}J,I)\in AUT(Q,\cdot) \\ \Leftrightarrow
 (\varrho,(J\alpha J)^{-1},I)\in AUT(Q,\cdot)\Leftrightarrow\varrho'=J\alpha J\Leftrightarrow\alpha=J\varphi J~\textrm{for some}~\varphi\in\Psi (Q,\cdot).
\end{gather*}
\begin{gather*}
R_{\sigma (x)}J\in R_{[\sigma(x)]\alpha^{-1}}J\Psi(Q,\cdot)\Leftrightarrow R_{\sigma (x)}J= R_{[\sigma(x)]\alpha^{-1}}J\varphi~\textrm{for some}~
\varphi\in\Psi(Q,\cdot)\Leftrightarrow \\ (y\cdot \sigma (x))^{-1}=\big(\big[y\cdot [\sigma(x)]\alpha^{-1}\big]^{-1}\big)\varphi\Leftrightarrow
(I,\alpha,J\varphi J,)\in AUT(Q,\cdot)\Leftrightarrow \\ (J\varphi J,J\alpha J,I)\in AUT(Q,\cdot)\Leftrightarrow \big(J\varphi J,(J\alpha^{-1} J)^{-1},I\big)\in AUT(Q,\cdot)\Leftrightarrow \\ J\varphi J\in\Phi(Q,\cdot)~\textrm{and}~J\alpha^{-1} J\in\Psi(Q,\cdot)\Leftrightarrow
 \varrho=J\varphi J~\textrm{for some}~\varrho\in\Phi (Q,\cdot)\\~\textrm{and}~\alpha=J\varphi J~\textrm{for some}~\varphi\in\Psi (Q,\cdot).\qed
\end{gather*}
\end{proof}

\begin{mycor}\label{15}
Let $(Q,\cdot)$ be a RIPL with a bijective self map $\sigma$
 and let $(H,\circ )$ be the A-holomorph of $(Q,\cdot)$ with a self map $\sigma'$ such that $\sigma'~:~(\alpha,x)\mapsto (\alpha,\sigma(x))$ for all $(\alpha,x)\in H$. If $(H,\circ)$ is a $\sigma'$-GBL, then
\begin{displaymath}
A(Q)=\big\{\sigma \rho\sigma^{-1},\rho,\sigma J\varphi(\sigma J)^{-1},J\varphi J|~\textrm{for some}~\rho\in\mathcal{P}(Q,\cdot)~\textrm{and some}~\varphi\in\Psi (Q,\cdot)\big\}.
 \end{displaymath}
\end{mycor}
\begin{proof}
Use Theorem~\ref{14}.\qed
\end{proof}

 \begin{mycor}\label{16}
Let $(Q,\cdot)$ be a RIPL and let $(H,\circ )$ be the A-holomorph of $(Q,\cdot)$. The following are equivalent
 \begin{enumerate}
 \item $(H,\circ)$ is a Bol loop.
\item
\begin{enumerate}
   \item $(Q,\cdot)$ is a Bol loop;
  \item $\alpha,\gamma\in\mathcal{P}(Q,\cdot)$ for all $\alpha,\gamma\in A(Q)$;
  \end{enumerate}
\item
\begin{enumerate}
\item $(Q,\cdot)$ is a Bol loop;
  \item $\alpha,\gamma\in J\Psi(Q,\cdot)J$ for all $\alpha,\gamma\in A(Q)$;
   \item $\varrho =J\varphi J$ for some $\varphi\in\Psi(Q,\cdot)$ and some $\varrho\in\Phi(Q,\cdot)$.
 \end{enumerate}
 \end{enumerate}
\end{mycor}
\begin{proof}
Apply Theorem~\ref{14} with $\sigma=I$.\qed
\end{proof}

\begin{mycor}\label{17}
Let $(Q,\cdot)$ be a RIPL and let $(H,\circ )$ be the A-holomorph of $(Q,\cdot)$.
If $(H,\circ)$ is a Bol loop, then
\begin{displaymath}
A(Q)=\big\{\rho,J\varphi J|~\textrm{for some}~\rho\in\mathcal{P}(Q,\cdot)~\textrm{and some}~\varphi\in\Psi (Q,\cdot)\big\}.
 \end{displaymath}
\end{mycor}
\begin{proof}
Use Corollary~\ref{15} with $\sigma=I$.\qed
\end{proof}

\begin{myth}\label{18}
Let $(Q,\cdot)$ be a RIPL with identity element $e$ and a self map $\sigma$
 and let $(H,\circ )$ be the A-holomorph of $(Q,\cdot)$ with a self map $\sigma'$ such that $\sigma'~:~(\alpha,x)\mapsto (\alpha,\sigma(x))$ for all $(\alpha,x)\in H$. Let
\begin{displaymath}
\psi : \mathcal{P}(Q,\cdot)\rightarrow N_{\rho}(Q,\cdot)\uparrow \psi(U)=eU,~\phi :\Phi(Q,\cdot)\rightarrow \Psi(Q,\cdot)~\uparrow \phi(U)=U',
\end{displaymath}
\begin{displaymath}
\varpi :\Phi(Q,\cdot)\rightarrow N_{\mu}(Q,\cdot)\uparrow \varpi(U)=eU~ \textrm{and}~ \beta: \Psi(Q,\cdot)\rightarrow N_{\mu}(Q,\cdot)\uparrow \beta(U')=eU'.
\end{displaymath}If $(H,\circ)$ is a $\sigma'$-GBL, then
 \begin{enumerate}
   \item $R_{\sigma (x)}^{-1}R_{\sigma(x\gamma^{-1})}\stackrel{\psi,\varpi}{\cong}\sigma (x)^{-1}\sigma(x\gamma^{-1})~\forall~\gamma\in A(Q)~,x\in Q$.
   \item $R_{\sigma (x)}^{-1}R_{[\sigma(x)]\alpha^{-1}}\stackrel{\psi,\varpi}{\cong}\sigma (x)^{-1}[\sigma(x)]\alpha^{-1}~\forall~\alpha\in A(Q)~,x\in Q$.
   \item $JR_{\sigma(x\gamma^{-1})}^{-1}R_{\sigma (x)}J\stackrel{\beta}{\cong}\sigma (x)^{-1}\sigma(x\gamma^{-1})~\forall~\gamma\in A(Q)~,x\in Q$.
   \item $JR_{[\sigma(x)]\alpha^{-1}}^{-1}R_{\sigma (x)}J\stackrel{\beta}{\cong}\sigma (x)^{-1}[\sigma(x)]\alpha^{-1}~\forall~\alpha\in A(Q)~,x\in Q$.
   \item $R_{\sigma (x)}^{-1}R_{\sigma(x\gamma^{-1})}\stackrel{\phi}{\cong}R_{\sigma(x\gamma^{-1})}^{-1}R_{\sigma (x)}J~\forall~\gamma\in A(Q)~,x\in Q$.
   \item $R_{\sigma (x)}^{-1}R_{[\sigma(x)]\alpha^{-1}}\stackrel{\phi}{\cong}JR_{[\sigma(x)]\alpha^{-1}}^{-1}R_{\sigma (x)}J~\forall~\alpha\in A(Q)~,x\in Q$.
 \end{enumerate}
\end{myth}
\begin{proof}
This is achieved by using Theorem~\ref{10}, Corollary~\ref{8} and Theorem~\ref{sim}.\qed
\end{proof}

\begin{myth}\label{19}
Let $(Q,\cdot)$ be a RIPL with a self map $\sigma$
 and let $(H,\circ )$ be the A-holomorph of $(Q,\cdot)$ with a self map $\sigma'$ such that $\sigma'~:~(\alpha,x)\mapsto (\alpha,\sigma(x))$ for all $(\alpha,x)\in H$. If $(H,\circ)$ is a $\sigma'$-GBL, then
 \begin{enumerate}
\item the correspondence
\begin{diagram}
&                                                                               & \sigma (x)^{-1}\sigma(x\gamma^{-1})\\
&\ruMapsto^{\psi,\varpi}_{}&\uMapsto^{\beta}_{\textrm{isomorphism}}\\
R_{\sigma (x)}^{-1}R_{\sigma(x\gamma^{-1})} &\rMapsto^{\phi}_{\textrm{isomorphism}} & JR_{\sigma(x\gamma^{-1})}^{-1}R_{\sigma (x)}J
\end{diagram}
is true for all $\gamma\in A(Q)$ and $x\in Q$, $\psi =\phi\beta$ and $ \varpi=\phi\beta$.
\item the correspondence
\begin{diagram}
&                                                                               & \sigma (x)^{-1}[\sigma(x)]\alpha^{-1}\\
&\ruMapsto^{\psi,\varpi}_{}&\uMapsto^{\beta}_{\textrm{isomorphism}}\\
R_{\sigma (x)}^{-1}R_{[\sigma(x)]\alpha^{-1}} &\rMapsto^{\phi}_{\textrm{isomorphism}} & JR_{[\sigma(x)]\alpha^{-1}}^{-1}R_{\sigma (x)}J
\end{diagram}
is true for all $\alpha\in A(Q)$ and $x\in Q$, $\psi =\phi\beta$ and $ \varpi=\phi\beta$.
\item the commutative diagram
\begin{diagram}
 \mathcal{P}(L,\cdot)&\rTo^{\psi}_{}    & N_\mu(L,\cdot)\\
\dDotsto^{\delta_1}_{}&\ruTo^{\varpi}_{}&\uTo^{\beta}_{\textrm{isomorphism}}\\
\Phi(L,\cdot) &\rTo^{\phi}_{\textrm{isomorphism}} & \Psi(L,\cdot)
\end{diagram}
is true, $\delta_1=\psi\beta^{-1}\phi^{-1} =\psi\varpi^{-1}$ and $R_{\sigma (x)}^{-1}R_{\sigma(x\gamma^{-1})}\stackrel{\delta_1}{\cong}R_{\sigma (x)}^{-1}R_{\sigma(x\gamma^{-1})}$ for all $\gamma\in A(Q)$ and $x\in Q$.
\item the commutative diagram
\begin{diagram}
 \mathcal{P}(L,\cdot)&\rTo^{\psi}_{}    & N_\mu(L,\cdot)\\
\uDotsto^{\delta_2}_{}&\ruTo^{\varpi}_{}&\uTo^{\beta}_{\textrm{isomorphism}}\\
\Phi(L,\cdot) &\rTo^{\phi}_{\textrm{isomorphism}} & \Psi(L,\cdot)
\end{diagram}
is true, $\varpi=\phi\beta= \delta_{2}\psi$ and $R_{\sigma (x)}^{-1}R_{[\sigma(x)]\alpha^{-1}}\stackrel{\delta_2}{\cong}R_{\sigma (x)}^{-1}R_{[\sigma(x)]\alpha^{-1}}$ for all $\alpha\in A(Q)$ and $x\in Q$.
\end{enumerate}
\end{myth}
\begin{proof}
The proof follows from Theorem~\ref{18} and Theorem~\ref{sim}.\qed
\end{proof}


\begin{thebibliography}{99}
\bibitem{phd142} Adeniran J. O. (2002), {\it The study of properties of
certain class of loops via their Bryant-Schneider group}, Ph.D.
thesis, University of Agriculture, Abeokuta.
\bibitem{phd111} Adeniran J. O. (1997), {\it On generalised Bol
loop identity and related identities}, M.Sc. thesis, Obafemi Awolowo
University, Ile-Ife.
\bibitem{phd79} Adeniran J. O. (2005), {\it On holomorphic theory of a
class of left Bol loops}, Scientific Annal
of A.I.I Cuza Univ., 51, 1, 23--28.
\bibitem{1} Adeniran J.O. and Akinleye, S. A. (2001), {\it On some loops
satisfying the generalised Bol identity}, Nig. Jour. Sc. 35, 101--107.
\bibitem{ho1} Adeniran J. O., Jaiyeola T. G. and Idowu K. A. (2014), {\it Holomorph of generalized Bol loops}, Novi Sad Journal of Mathematics, 44 (1), 37--51.
\bibitem{2} Adeniran J. O. and Solarin, A.R.T. (1997), {\it A
note on automorphic inverse property loops}, Zbornik Radfova, Coll.
of Sci. papers, 20, 47--52.
\bibitem{3} Adeniran J. O. and Solarin,
A.R.T. (1999), {\it A note on generalised Bol Identity}, Scientific Annal
of A.I.I Cuza Univ., 45 (1), 19--26.
\bibitem{5} Ajmal N. (1978), {\it A
generalisation of Bol loops}, Ann. Soc. Sci. Bruxelles Ser. 1,
92(4), 241--248.
\bibitem{belousov} Belousov V. (1967), "The Foundations of the Theory of Quasigroups and Loops", Moscow, Nauka (Russian).
\bibitem{6} Blaschke W. and Bol G. (1938),
Geometric der Gewebe, Springer Verlags.
\bibitem{16} Bol G., Gewebe and Gruppen (1937), {\it On Bol
loops}, Math. Ann. 144, 414--431.
\bibitem{7} Bruck R. H. (1944), {\it Contributions to the theory of Loops}, Trans. Amer. Soc. 55,
245--354.  http://dx.doi.org/10.1090/s0002-9947-1946-0017288-3
\bibitem{8} Bruck R. H. (1971), {\it A survey of binary systems}, Springer-Verlag, Berlin-G\"ottingen-Heidelberg, 185pp.  http://dx.doi.org/10.1007/978-3-662-43119-1
\bibitem{phd40} R. H. Bruck and L. J. Paige (1956), {\it Loops whose
inner mappings are automorphisms}, The annals of Mathematics, 63,
2, 308--323.  http://dx.doi.org/10.2307/1969612
\bibitem{bur} Burn R. P. (1978), {\it Finite Bol loops}, Math. Proc. Camb. Phil. Soc. 84, 377--385. http://dx.doi.org/10.1017/s0305004100055213
\bibitem{bur1} Burn R. P. (1981), {\it Finite Bol loops II}, Math. Proc. Camb. Phil. Soc. 88, 445--455. http://dx.doi.org/10.1017/s0305004100058357
\bibitem{bur2} Burn R. P. (1985), {\it Finite Bol loops III}, Math. Proc. Camb. Phil. Soc. 97, 219--223. http://dx.doi.org/10.1017/s0305004100062770
\bibitem{phd92} Bryant B. F. and Schneider H. (1966), {\it Principal
loop-isotopes of quasigroups}, Canad. J. Math. 18, 120--125. http://dx.doi.org/10.4153/cjm-1966-016-8
\bibitem{chein1} Chein O. and Goodaire E. G. (2008), {\it Bol loops with a large left nucleus}, Comment. Math. Univ. Carolin. 49, 2, 171--196.
\bibitem{chein2} Chein O. and Goodaire E. G. (2007), {\it Bol loops of nilpotence class two}, Canad. J. Math. 59, 2, 296--310.  http://dx.doi.org/10.4153/cjm-2007-012-7
\bibitem{gg22} Chein O. and Goodaire E. G. (2005), {\it  A new
construction of Bol loops: the "odd" case}, Quasigroups and
Related Systems, 13, 1, 87--98.
\bibitem{phd94} Chiboka V. O. (1996), {\it The Bryant-Schneider group of an extra loop}, Collection of Scientific papers of the Faculty of Science, Kragujevac, 18, 9--20.
\bibitem{phd80} Chiboka V. O. and A. R. T. Solarin (1991), {\it Holomorphs of conjugacy closed loops}, Scientific Annals of Al.I.Cuza. Univ. 37, 3, 277--284.
\bibitem{phd127} Foguel T., Kinyon M. K. and Phillips J. D. (2006), {\it On
twisted subgroups and Bol loops of odd order}, Rocky Mountain J.
Math. 36, 183-–212.  http://dx.doi.org/10.1216/rmjm/1181069494
\bibitem{phd44} Huthnance Jr. E. D. (1968), {\it A theory of
generalised Moufang loops}, Ph.D. thesis, Georgia Institute of
Technology.
\bibitem{davidref:10} Jaiyeola, T. G. (2009), {\it A study of new concepts in smarandache quasigroups and
loops}, ProQuest Information and Learning(ILQ), Ann Arbor, USA,
127pp.
\bibitem{phd115} Kinyon M. K. and Phillips J. D. (2004), {\it
Commutants of Bol loops of odd order}, Proc. Amer. Math. Soc. 132,
617--619.
\bibitem{phd121} Kinyon M. K., Phillips J. D. and Vojt\v echovsk\'y P. (2008), {\it When is the commutant of a Bol loop a
subloop?}, Trans. Amer. Math. Soc., 360, 2393--2408.  http://dx.doi.org/10.1090/s0002-9947-07-04391-7
\bibitem{19} Moufang R. (1935), {\it On Quasigroups}, Zur
Struktur von Alterntivkorpern, Math. Ann. 110, 416--430.
\bibitem{nagy1} Nagy G. P. (2009), {\it A class of finite simple Bol loops of exponent 2},
Trans. Amer. Math. Soc. 361 , 5331-5343. http://dx.doi.org/10.1090/s0002-9947-09-04646-7
\bibitem{nagy2} Nagy G. P. (2008), {\it A class of simple proper Bol loop}, Manuscripta Mathematica, 127(1), 81--88. http://dx.doi.org/10.1007/s00229-008-0188-5
\bibitem{nagy3} Nagy G. P. (2008), {\it Some remarks on simple Bol loops}, Comment. Math. Univ. Carolin. 49(2), 259--270.
\bibitem{25} Robinson D. A. (1964), {\it Bol loops}, Ph.D Thesis, University of Wisconsin, Madison.
\bibitem{phd7} Robinson D. A. (1971), {\it Holomorphic theory of
extra loops}, Publ. Math. Debrecen 18, 59--64.
\bibitem{phd93} Robinson D. A. (1980), {\it The Bryant-Schneider group of a loop}, Extract Des Ann. De la Socii\'et \'e Sci. De Brucellaes, Tome 94, II-II, 69--81.
\bibitem{30} Sharma B. L. (1976), {\it Left
loops which Satisfy the left Bol identity}, Proc. Amer. Math. Soc.,
61, 189-195.  http://dx.doi.org/10.1090/s0002-9939-1976-0422480-4
\bibitem{31} Sharma B. L. (1977), {\it Left Loops which
satisfy the left Bol identity (II)}, Ann. Soc. Sci. Bruxelles, Sér. I, 91, 69--78.
\bibitem{32} Sharma B. L. and Sabinin,
L.V. (1979), {\it On the Algebraic properties of half Bol Loops},
Ann. Soc. Sci. Bruxelles Sér. I, 93, 4, 227--240.
\bibitem{33} Sharma B. L. and Sabinin, L. V. (1976), {\it On the existence of Half
Bol loops}, Scientific Annal
of A.I.I Cuza Univ., 22, 2, 147--148.
\bibitem{34} Solarin A.R.T. and Sharma B. L. (1981),
{\it On the Construction of Bol loops}, Scientific Annal
of A.I.I Cuza Univ., 27, 1, 13--17.
\bibitem{32.1} Solarin A.R.T. and Sharma B.L. (1984), {\it Some examples of Bol loops}, Acta
Carol, Math. and Phys., 25(1), 59--68.
\bibitem{35} Solarin A.R.T. and Sharma B.L. (1984), {\it On the Construction of Bol loops II}, Scientific Annal
of A.I.I Cuza Univ., 30, 2, 7--14.
\bibitem{37} Solarin A.R.T. (1986), {\it Characterization of Bol loops
of small orders}, Ph.D. Dissertation, Universiy of Ife.
\bibitem{38} Solarin A.R.T. (1988), {\it On the Identities of Bol Moufang Type},
Kyungpook Math., 28(1), 51--62.
\end{thebibliography}
\end{document}